\documentclass{amsart}

\usepackage{amssymb}
\usepackage{amsthm}
\usepackage{amsmath}
\usepackage[matrix, arrow]{xy}
\xyoption{arrow}
\theoremstyle{plain}\newtheorem{Theorem}{Theorem}
\theoremstyle{plain}
\theoremstyle{plain}\newtheorem{Corollary}[Theorem]{Corollary}
\theoremstyle{plain}\newtheorem{Lemma}[Theorem]{Lemma}
\theoremstyle{plain}
\theoremstyle{definition}
\theoremstyle{definition}
\theoremstyle{definition}
\theoremstyle{definition}
\theoremstyle{definition}
\theoremstyle{definition}\newtheorem{Remark}[Theorem]{Remark}

 \def\OG{{\mathcal{O}G}}

\def\CO{{\mathcal{O}}}

            \def\tenk{\otimes_k}

\def\dim{\mathrm{dim}}

      \def\tenO{\otimes_{\mathcal{O}}}

\def\rk{\mathrm{rk}}        
\def\Res{\mathrm{Res}}

\makeindex
\title[Vertices of tensor products]{A note on vertices of 
indecomposable tensor products} 
\author[Linckelmann]{Markus Linckelmann}
\address{Department of Mathematics,
City, University of  London EC1V 0HB, United Kingdom}
\email{markus.linckelmann.1@city.ac.uk}



\date{\today}

\begin{document}

\begin{abstract}
G. Navarro raised the question under what circumstancs two vertices of 
two indecomposable modules over a finite group algebra generate a Sylow 
$p$-subgroup. The present note provides a sufficient criterion for when 
this is the case. This generalises a result by Navarro for simple 
modules over finite $p$-solvable groups, which is the main motivation 
for this note. 
\end{abstract}

\maketitle

Let $p$ be a prime and $\CO$ a complete local principal ideal domain 
with residue field $k$ of characteristic $p$. We allow the case 
$\CO=k$, unless stated otherwise. We assume that $k$ is large enough for 
the modules and their sources that appear in this note to be absolutely 
indecomposable; the point of this hypothesis is that it ensures that 
indecomposable modules over finite group algebras have multiplicity 
modules (see \cite[\S 5.7]{LiBookI} for background material). 
If the field of fractions $K$ of $\CO$ has characteristic zero, then
we also assume that $K$ is large enough so that the irreducible
characters over $K$ that arise below are absolutely irreducible. 
In some of the results in the literature cited below, $k$ is 
algebraically closed for convenience, but it is easy to see that 
`large enough' in the sense above will do for the quoted results. 
Modules are finitely generated left modules. 

\begin{Theorem} \label{tensor-indec}
Let $G$ be a finite group and let $U$, $V$ be indecomposable
$\CO$-free $\OG$-modules having sources of $\CO$-ranks prime to $p$.
Suppose that $W=U\tenO V$ is indecomposable and that $W$
has a simple multiplicity module. Then the sources of $W$ have
$\CO$-rank prime to $p$, and there exist vertices $Q$, $R$ of $U$, $V$, 
respectively such that $Q\cap R$ is a vertex of $W$ and such that
$QR$ is a Sylow $p$-subgroup of $G$. 
\end{Theorem}

The proof of Theorem \ref{tensor-indec} yields some further technical
information regarding which intersections of vertices of $U$ and $V$
do yield vertices of $W$ (cf. Remark \ref{rem1}), and some information 
regarding the $p$-parts of the ranks of $U$, $V$, $W$ and the dimensions 
of their multiplicity modules (cf. Remark \ref{rem2}). The proof shows 
also that the hypothesis on the simplicity of the multiplicity module of 
$W$ can be replaced by a slightly weaker condition on the dimension of 
the multiplicity module of $W$ (cf. Remark \ref{rem3}). See K\"ulshammer 
\cite[Proposition 2.1]{Kuls93} for a sufficient criterion for when the 
tensor product of two modules is indecomposable. 
We note some immediate consequences of Theorem \ref{tensor-indec}. 
By a result of Kn\"orr (in the proof of \cite[Proposition 3.1]{Knorr}, 
also described in \cite[Corollary 5.7.9]{LiBookI}), simple 
$kG$-modules and $\OG$-lattices with irreducible characters have simple 
multiplicity modules. Thus Theorem \ref{tensor-indec} implies the 
following two results.

\begin{Corollary} \label{tensor-simple}
Let $G$ be a finite group and let $U$, $V$ be simple $kG$-modules
with sources of dimensions prime to $p$ such that $W=$ $U\tenk V$
is simple. Then the sources of $W$ have dimension prime to $p$, and
there exist vertices $Q$, $R$ of $U$, $V$, respectively, such that 
$Q\cap R$ is a vertex of $W$ and such that $QR$ is a Sylow $p$-subgroup 
of $G$.
\end{Corollary}

\begin{Corollary} \label{tensor-irred}
Suppose that the field of fractions $K$ of $\CO$ has characteristic
zero. Let $U$, $V$ be $\CO$-free $\OG$-modules with irreducible
characters and sources of $\CO$-ranks prime to $p$. Suppose that
the character of $W=U\tenO V$ is irreducible. 
Then the sources of $W$ have $\CO$-rank prime to $p$, and
there exist vertices $Q$, $R$ of $U$, $V$, respectively, such that 
$Q\cap R$ is a vertex of $W$ and such that $QR$ is a Sylow $p$-subgroup 
of $G$.
\end{Corollary}

By a result of Puig in unpublished notes from 1988 (see  e. g. 
\cite[Theorem (30.5)]{Thev} or \cite[Theorem 10.6.8]{LiBookI}), if 
$G$ is a finite $p$-solvable group, then simple $kG$-modules have 
endopermutation sources; in particular, their sources have dimensions 
prime to $p$. Thus Corollary \ref{tensor-simple} implies  the following 
result due to Navarro (and this is the main motivation for this note).

\begin{Corollary}[{Navarro \cite[Theorem A]{Nav2019}}] 
\label{tensor-psolvable}
Let $G$ be a finite $p$-solvable group and let $U$, $V$ be
simple $kG$-modules such that $W=$ $U\tenk V$ is simple.
Then there exist vertices $Q$, $R$ of $U$, $V$, respectively,
such that $Q\cap R$ is a vertex of $W$ and such that $QR$ is
a Sylow $p$-subgroup of $G$.
\end{Corollary}

For the proof of Theorem \ref{tensor-indec}, we collect a few
elementary observations on vertices and sources which imply that the 
first two statements in Theorem \ref{tensor-indec} hold under weaker 
hypotheses. We refer to \cite[\S 5.1]{LiBookI} for definitions and
basic properties of Green's theory of vertices and sources of modules
from \cite{Green59}; these - as well as a number of arguments in this 
note - depend on the fact, used without further mention, that the 
Krull-Schmidt theorem holds in the context of the present note. The 
following is well-known.

\begin{Lemma}[see e. g. {\cite[Theorem 5.1.11]{LiBookI}} 
or {\cite[Ch. 3, Theorem 1.17]{NaTs}} ] \label{lemma1}
Let $G$ be a finite group and let $U$, $V$ be indecomposable
$\CO$-free $\OG$-modules, and let $Q$, $R$ be vertices of $U$, $V$,
respectively. Then for every indecomposable direct summand $W$ of
$U\tenO V$ there is $x\in$ $G$ such that $Q\cap {^xR}$ contains a
vertex of $W$.
\end{Lemma}

\begin{Lemma} \label{lemma2}
Let $G$ be a finite group and let $U$, $V$ be indecomposable
$\CO$-free $\OG$-modules having sources of $\CO$-ranks prime to $p$.

\begin{enumerate}
\item[{\rm (i)}]
For any $x\in$ $G$ there exists an indecomposable direct summand
$W$ of $U\tenO V$ such that $Q\cap {^xR}$ is contained in a vertex
of $W$.

\item[{\rm (ii)}]
If $x\in G$ is chosen such that $Q\cap {^xR}$ has maximal order amongst
the subgroups of the form $Q\cap {^yR}$, where $y\in G$, then 
$U\tenO V$ has an indecomposable direct summand $W$ with vertex 
$Q\cap {^xR}$ and with sources of $\CO$-rank prime to $p$.

\item[{\rm (iii)}]
Suppose that $U\tenO V$ is indecomposable. Then the sources of 
$U\tenO V$ have $\CO$-rank prime to $p$, and for any $x\in$ $G$, the
subgroup $Q\cap {^xR}$ is contained in a vertex of $U\tenO V$. If 
$Q\cap {^xR}$ has maximal order amongst all subgroups of this form, 
then $Q\cap {^xR}$ is a vertex of $U\tenO V$.
\end{enumerate}
\end{Lemma}

\begin{proof}
Let $Q$, $R$ be vertices of $U$, $V$, respectively. Let $x\in$ $G$.
Then $Q\cap {^xR}$ is contained in a vertex of $U$ and in a vertex of
$V$. The assumptions on $U$ and $V$ imply that $\Res^G_{Q\cap {^xR}}(U)$
and $\Res^G_{Q\cap {^xR}}(V)$ have indecomposable direct summands of
$\CO$-ranks prime to $p$. Thus their tensor product 
$\Res^G_{Q\cap {^xR}}(U\tenO V)$ has an idecomposable direct summand $Y$ 
of $\CO$-rank prime to $p$. In particular, $Q\cap {^xR}$ is the vertex
of $Y$. Moreover, since $Y$ is indecompsable, it follows that $Y$ is 
isomorphic to a direct summand of $\Res^G_{Q\cap {^xR}}(W)$ for some 
indecomposable direct summand $W$ of $U\tenO V$, and hence $Q\cap {^xR}$ 
is contained in a vertex of $W$. This shows (i). The same argument,
assuming in addition that $Q\cap {^xR}$ has maximal order amongst the
subgroups of this form, in conjunction with Lemma \ref{lemma1} shows 
(ii). Statement (iii) follows from (i) and (ii).
\end{proof}

For $m$ a positive integer, we denote by $m_p$ the highest power of $p$ 
which divides $m$, and we denote by $\rk_\CO(U)$ the $\CO$-rank of a 
free $\CO$-module $U$. 

\begin{Lemma}[{Green \cite[Theorem 9]{Green59}}] \label{lemma3}
Let $G$ be a finite group, $U$ an $\CO$-free indecomposable
$\OG$-module, and $Q$ a vertex of $U$. Then $\rk_\CO(U)_p\geq$
$|G:Q|_p$.
\end{Lemma}

See e. g. \cite[Theorem 5.12.13]{LiBookI} or 
\cite[Ch. 7, Theorem 7.5]{NaTs} for proofs of this Lemma. 
The key ingredient for the proof of Theorem \ref{tensor-indec} is the
following result due to Kn\"orr, which is a criterion for when the 
inequality in Lemma \ref{lemma3} is an equality. 

\begin{Theorem}[{\cite[Theorem 4.5]{Knorr}}] \label{Knoerr45}
Let $G$ be a finite group, $U$ an indecomposable $\CO$-free 
$\OG$-module, and $Q$ a vertex of $U$. Suppose that the sources
of $U$ have $\CO$-rank prime to $p$ and that $U$ has a simple
multiplicity module. Then $\rk_\CO(U)_p=$ $|G:Q|_p$.
\end{Theorem}

The statement in \cite[Theorem 4.5]{Knorr} does not mention 
multiplicity modules explicitly, but the hypotheses in the theorem 
ensure the simplicity of multiplicity modules, and this is all that is 
used in the proof of \cite[Theorem 4.5]{Knorr}. See 
\cite[\S 9]{ThevdualityG} for more general background material on 
multiplicity modules and characterisations of simple multiplicity 
modules. A description of some of this material closer to the 
terminology used above is given in 
\cite[Theorems 5.7.7, 5.12.15]{LiBookI}.

\begin{proof}[{Proof of Theorem \ref{tensor-indec}}]
Let $Q$, $R$ be vertices of $U$, $V$, respectively, chosen such that
$Q\cap R$ has maximal order amongst all intersections of a vertex of $U$ 
and a vertex of $V$. It follows from Lemma \ref{lemma2} (iii) that 
$Q\cap R$ is a vertex of $W=$ $U\tenO V$ and that the sources of $W$ 
have $\CO$-rank prime to $p$.

Since $W$ has a simple multiplicity module, it follows from Theorem 
\ref{Knoerr45} and Lemma \ref{lemma3} that we have
$$|G: Q\cap R|_p=\rk_\CO(W)_p=\rk_\CO(U)_p\cdot\rk_\CO(V)_p \geq
|G:Q|_p\cdot |G:R|_p\ .$$ 
The left side is also equal to $|G:Q|_p\cdot |Q:Q\cap R|$, hence
cancelling $|G:Q|_p$ yields 
$$|Q:Q\cap R| \geq |G:R|_p = |G|_p/|R|\ .$$
Now $|Q:Q\cap R|=$ $|QR|/|R|$, so together this yields $|QR|\geq$ 
$|G|_p$.

Let $P$ be a Sylow $p$-subgroup of $G$ 
containing $Q$. Let $x\in$ $G$ such that ${^x{R}}\subseteq$ $P$. 
It follows from Lemma \ref{lemma2} (ii) that $Q\cap {^x{R}}$ is contained
in a vertex of $W$. In particular, we have $|Q\cap {^x{R}}|\leq$
$|Q\cap R|$. Thus $|Q({^xR})|\geq |QR| \geq$ $|P|$. 
But since both $Q$, ${^xR}$ are contained in $P$, we also have
$|Q({^xR})|\leq |P|$. Thus all inequalities in this proof are 
equalities. This forces $Q({^xR})=P$ and $|Q\cap {^xR}|=$
$|Q\cap R|$, so $Q$ and ${^xR}$ (instead of $R$) satisfy all
conclusions. 
\end{proof}

\begin{Remark} \label{rem1}
The proof of Theorem \ref{tensor-indec} shows a bit more: with the
notation and hypotheses of Theorem \ref{tensor-indec}, any intersection
of a vertex of $U$ and a vertex of $V$ is contained in a vertex of $W$,
and any intersection of maximal order of a vertex of $U$ and a vertex of 
$V$ is a vertex of $W$. Moreover, for any choice of vertices  $Q$, $R$ 
of $U$, $V$, respectively, such that both $Q$, $R$ are contained in a 
Sylow $p$-subgroup $P$ of $G$, we have $P=$ $QR$, and $Q\cap R$ is a 
vertex of $W$. This points to some information about 
fusion in $G$: if $x\in$ $G$ is chosen such that $Q\cap {^xR}$ has 
maximal order amongst all subgroups of this form, then $Q\cap {^xR}$ is
a vertex of $W$, and by Lemma \ref{lemma2} (iii), for any $y\in$ $G$, the
group $Q\cap {^yR}$ is $G$-conjugate to a subgroup of $Q\cap {^xR}$. In
particular, if $x$, $y$ are two elements in $G$ such that $Q\cap {^xR}$ 
and $Q\cap {^yR}$ both have the same maximal order amongst all groups
of this form, then $Q\cap {^xR}$ and $Q\cap {^yR}$ are $G$-conjugate,
since they both are vertices of $W$. 
\end{Remark}

\begin{Remark} \label{rem2}
With the notation and hypotheses of Theorem \ref{tensor-indec}, denote
by $(Q,X)$, $(R,Y)$, $(S,Z)$ vertex-source pairs of $U$, $V$, $W$,
respectively. The fact that the displayed inequalities in the proof of 
Theorem \ref{tensor-indec} are all equalities implies that we have 
$$|G:Q|_p= \rk_\CO(U)_p\ ,$$  
$$|G:R|_p= \rk_\CO(V)_p\ ,$$ 
$$|G:S|_p = \rk_\CO(W)_p\ .$$
Moreover, by \cite[Theorem 5.12.15]{LiBookI} the associated multiplicity 
modules $M_U$, $M_V$, $M_W$ of $U$, $V$, $W$, respectively, satisfy
$$\dim_k(M_U)_p = | N_G(Q,X)/Q |_p\ ,$$
$$\dim_k(M_V)_p = | N_G(R,Y)/R |_p\ ,$$
$$\dim_k(M_W)_p = | N_G(S,Z)/S |_p\ .$$
\end{Remark}

\begin{Remark} \label{rem3}
Theorem \ref{tensor-indec} holds with slightly weaker hypotheses on the 
multiplicity modules of $W$. Instead of requiring the simplicity of a 
multiplicity module of $W$, it suffices to require the equality
$$\dim_k(M_W)_p = | N_G(S,Z)/S |_p $$
for the multiplicity module $M_W$ of $W$ associated with a vertexs-souce
pair $(S,Z)$ of $W$ (this is the last equality in Remark 
\ref{rem2}). By a result of Kn\"orr \cite[Proposition 4.2]{Knorr}, this 
equality holds if $M_W$ is simple (essentially because $M_W$ is then a 
projective simple module of a finite central $p'$-extension of 
$N_G(S,Z)/S$), but the simplicity of $M_W$ is not necessary in general 
for this equality. In particular, even if $M_W$ is simple, we do not 
know whether this implies that $M_U$, $M_V$ are necessarily simple as 
well. 
\end{Remark}

\begin{Remark} \label{rem4}
With the notation and hypotheses of Theorem \ref{tensor-indec}, if $D$, 
$E$ are defect groups of the blocks to which $U$, $V$ belong, 
respectively, such that $D$, $E$ are both contained in a fixed Sylow 
$p$-subgroup $P$ of $G$, then $D$ and $E$ contain vertices of $U$ and 
$V$, respectively, and hence $P=$ $DE$ by Remark \ref{rem1}. It would 
be interesting to investigate whether there are more precise 
relationships between the defect groups of the three blocks to which 
$U$, $V$, $W$ belong. For the remainder of this Remark, suppose that 
all three modules $U$, $V$, $W$ have simple multiplicity modules (this 
is for instance the case in any of the Corollaries
\ref{tensor-simple}, \ref{tensor-irred}, \ref{tensor-psolvable}). Then,
as a consequence of the results of Kn\"orr in \cite{Knorr}, the vertices 
$Q$, $R$, $Q\cap R$ of $U$, $V$, $W$ in  Theorem \ref{tensor-indec} are 
centric in the fusion systems of the blocks to which these modules 
belong (see \cite[Theorem 10.3.1]{LiBookII} for a formulation of 
\cite[Theorem 3.3]{Knorr} in terms of fusion systems of blocks). Thus, 
for instance, if the block of $\OG$ to which $U$ (resp. $V$, $W$) 
belongs has abelian defect groups, then $Q$ (resp. $R$, $Q\cap R$) is a 
defect groups of this block. Furthermore, by a slight generalisation 
\cite[Theorem 10.3.6]{LiBookII} of a result of Erdmann 
\cite{Erdmanncyclic}, if $Q$ (resp. $R$, $Q\cap R$) in Theorem 
\ref{tensor-indec} is cyclic, then it is a defect group of the block 
to which $U$ (resp. $V$, $W$) belongs. 
\end{Remark} 

\begin{Remark} \label{GiannelliRemark}
E. Giannelli pointed out that the situation is well understood for 
simple tensor products of simple modules over symmetric groups. Let $n$ 
be a positive integer. Let $U$, $V$ be simple 
$kS_n$-modules of dimensions greater than $1$ such that $W=$ $U\tenk V$ 
is simple. Then by \cite[Main Theorem]{BesKle} we have $p=2$ and $n$ is 
even. The exact tensor products which can arise in this way are 
described in \cite{Mor}, proving a conjecture of Gow and Kleshchev. It 
follows from that description, that then $n=2m$ for some odd integer $m$ 
and that $U$ can be chosen to be the basic spin module (labelled by the 
partition $(m+1,m-1)$). By \cite[Theorem 1.1]{DaKue}, the vertices of 
$U$ are the Sylow $2$-subgroups of $S_n$. It is, however, not always 
possible to choose vertices $Q$, $R$ of $U$, $V$, respectively, such 
that $QR$ is a Sylow $2$-subgroup {\it and} such that $Q\cap R$ is a 
vertex of $W$. Note that $U$ has sources of even dimension (cf. 
\cite[Lemma 5.3]{Ben88} and \cite[Theorem 6.1]{DaKue}), so the 
hypotheses of Theorem \ref{tensor-indec} are not satisfied. 
The following example is due to E. Giannelli. Consider the case $n=6$, 
with simple $kS_6$-modules $U$, $V$, $W$ labelled by the partitions 
$(4,2)$, $(5,1)$, $(3,2,1)$, respectively. It follows from 
\cite[Theorem 1.1]{Mor} that we have $W\cong$ $U\tenk V$. As mentioned
above,  \cite[Theorem 1.1]{DaKue} implies that $U$ has a Sylow 
$2$-subgroup $Q$ of $S_6$ as a vertex. Moreover, 
\cite[Theorem 1.2]{DaGia} implies that also $V$ has a  
Sylow $2$-subgroups $R$ of $S_6$ as a vertex. The product $QR$ is 
therefore a Sylow $2$-subgroup of $S_6$ if and only if $Q=R$. In that 
case we have $Q\cap R=Q=R$, but $W$ has a trivial vertex. 
\end{Remark}



\begin{thebibliography}{WWW}

\bibitem{Ben88} D. J. Benson, {\em Spin modules for symmetric groups.}
J. London Math. Soc. {\bf 38} (1988), 250--262.

\bibitem{BesKle} C. Bessenrodt and A. Kleshchev, {\em On tensor 
products of modular representations of symmetric groups.} Bull. 
London Math. Soc. {\bf 32} (2000), 292--296.

\bibitem{DaGia}  S. Danz and E. Giannelli, {\em Vertices of simple 
modules of symmetric groups labelled by hook partitions.} J. Group 
Theory {\bf 18} (2015), 313--334.

\bibitem{DaKue}  S. Danz and B.  K\"ulshammer, {\em The vertices and 
sources of the basic spin module for the symmetric group in 
characteristic $2$.} J. Pure Appl. Algebra {\bf 213} (2009), 
1264--1282.

\bibitem{Erdmanncyclic} K. Erdmann, {\em Blocks and simple modules with
cyclic vertices}, Bull. London Math. Soc. {\bf 9} (1977), 216--218.

\bibitem{Green59} J. A. Green, {\em On the indecomposable
representations of a finite group.} Math. Z. {\bf 70} (1959), 430--445.

\bibitem{Knorr} R. Kn\"orr, {\em On the vertices of irreducible
modules}, Ann. Math. {\bf 110} (1979), 487--499.

\bibitem{Kuls93} B. K\"ulshammer, {\em Some indecomposable modules and
their vertices.} J. Pure Appl. Algebra {\bf 86} (1993), 65--73.

\bibitem{LiBookI} M. Linckelmann, {\em The block theory of finite group
alegbras I}, Cambridge University Press, London Math. Soc. Student
Texts {\bf 91} (2018). 

\bibitem{LiBookII} M. Linckelmann, {\em The block theory of finite group
alegbras II}, Cambridge University Press, London Math. Soc. Student
Texts {\bf 92} (2018). 

\bibitem{Mor} L. Morotti, {\em Irreducible tensor products for symmetric 
groups in characteristic $2$},  Proc. Lond. Math. Soc. {\bf 116} (2018), 
1553--1598.

\bibitem{NaTs} H. Nagao and Y. Tsushima, {\em Representations of Finite
Groups}, Academic Press, San Diego (1989)

\bibitem{Nav2019} G. Navarro, {\em Products of characters, product of
subgroups.} Notes (2019). 

\bibitem{ThevdualityG} J. Th\'evenaz, {\em Duality in $G$-algebras},
Math. Z. {\bf 200} (1988) 47--85.

\bibitem{Thev} J. Th\'evenaz, {\em $G$-Algebras and Modular
Representation Theory}, Oxford Science Publications, Clarendon,
Oxford (1995).


\end{thebibliography}
\end{document}